\documentclass[preprint,10pt]{elsarticle}

\usepackage{amsmath}
\usepackage{amssymb}
\usepackage{amsthm}

\newtheorem{thm}{Theorem}[section]
\newtheorem{cor}[thm]{Corollary}
\newtheorem{lem}[thm]{Lemma}

\newtheorem{defin}[thm]{Definition}
\newtheorem{rem}[thm]{Remark}
\numberwithin{equation}{section}

\journal{}

\begin{document}

\begin{frontmatter}



\title{A new estimate for homogeneous fractional integral operators on the weighted Morrey space $L^{p,\kappa}$ when $\alpha p=(1-\kappa)n$}


\author{Jingliang Du and Hua Wang \footnote{In memory of Li Xue. E-mail address: 418419773@qq.com, wanghua@pku.edu.cn.}}
\address{School of Mathematics and Systems Science, Xinjiang University,\\
Urumqi 830046, P. R. China}

\begin{abstract}
For any $0<\alpha<n$, the homogeneous fractional integral operator $T_{\Omega,\alpha}$ is defined by
\begin{equation*}
T_{\Omega,\alpha}f(x)=\int_{\mathbb R^n}\frac{\Omega(x-y)}{|x-y|^{n-\alpha}}f(y)\,dy.
\end{equation*}
In this paper, we prove that if $\Omega$ satisfies certain Dini smoothness conditions on $\mathbf{S}^{n-1}$, then $T_{\Omega,\alpha}$ is bounded from $L^{p,\kappa}(w^p,w^q)$ (weighted Morrey space) to $\mathrm{BMO}(\mathbb R^n)$.
\end{abstract}


\end{frontmatter}

\section{Introduction}
\label{}
One of the most significant operators in harmonic analysis is the Riesz potential operator $I_{\alpha}$ of order $\alpha$, also known as the fractional integral operator. Let $\mathbb R^n$ be the $n$-dimensional Euclidean space endowed with the Lebesgue measure $dx$ and the Euclidean norm $|\cdot|$. Given $\alpha\in(0,n)$, the Riesz potential operator $I_{\alpha}$ of order $\alpha$ is defined by
\begin{equation*}
I_{\alpha}f(x):=\frac{1}{\gamma(\alpha)}\int_{\mathbb R^n}\frac{f(y)}{|x-y|^{n-\alpha}}dy,\quad x\in\mathbb R^n,
\end{equation*}
where $\gamma(\alpha)=\frac{2^{\alpha}\pi^{n/2}\Gamma(\alpha/2)}{\Gamma({(n-\alpha)}/2)}$ and $\Gamma(\cdot)$ being the usual gamma function. Let $(-\Delta)^{\alpha/2}$(under $0<\alpha<n$) denote the $\alpha/2$-th order Laplacian. Then $u=I_{\alpha}f$ is viewed as a solution of the $\alpha/2$-th order Laplace equation
\begin{equation*}
(-\Delta)^{\alpha/2}u=f
\end{equation*}
in the sense of the Fourier transform; i.e., $(-\Delta)^{\alpha/2}$ exists as the inverse of $I_{\alpha}$. It is well known that the Riesz potential operator $I_{\alpha}$ plays an important role in harmonic analysis, PDE and potential theory. The classical Hardy--Littlewood--Sobolev theorem states that $I_{\alpha}$ is bounded from $L^p(\mathbb R^n)$ to $L^q(\mathbb R^n)$ when $1<p<q<\infty$ and $1/q=1/p-\alpha/n$. For $p=1$, we know that $I_{\alpha}$ is bounded from $L^1(\mathbb R^n)$ to $WL^q(\mathbb R^n)$ when $0<\alpha<n$ and $q=n/{(n-\alpha)}$. Moreover, for the critical index $p=n/{\alpha}$, it is easy to verify that $I_{\alpha}$ is not bounded from $L^{n/{\alpha}}(\mathbb R^n)$ to $L^{\infty}(\mathbb R^n)$ (see \cite[p.119]{stein}). In this case, we also know that as a substitute the Riesz potential operator $I_{\alpha}$ is bounded from $L^{n/{\alpha}}(\mathbb R^n)$ to $\mathrm{BMO}(\mathbb R^n)$ (see \cite[p.130]{duoand}). For each locally integrable function $f$, closely related to the Riesz potential operator is the fractional maximal operator $M_{\alpha}$ on $\mathbb R^n$, which is defined by
\begin{equation*}
M_{\alpha}f(x):=\sup_{r>0}\frac{1}{|B(x,r)|^{1-\alpha/n}}\int_{|x-y|<r}|f(y)|\,dy,\quad x\in\mathbb R^n,
\end{equation*}
where the supremum is taken over all $r>0$. It can be shown that $M_{\alpha}$ satisfies the same norm inequalities as $I_{\alpha}$ for $0<\alpha<n$, $1<p<n/{\alpha}$ and $1/q=1/p-\alpha/n$. The weak type $(1,q)$ inequality can be proved by using covering lemma arguments when $q=n/{(n-\alpha)}$ (see, for example, \cite[Theorem 2]{muckenhoupt2}), and $M_{\alpha}$ is clearly of strong type $(p,\infty)$ when $p=n/{\alpha}$. Actually, by H\"{o}lder's inequality
\begin{equation*}
\frac{1}{|B(x,r)|^{1-\alpha/n}}\int_{|x-y|<r}|f(y)|\,dy\leq\bigg(\int_{|x-y|<r}|f(y)|^p\,dy\bigg)^{1/p}\leq\|f\|_{L^p}.
\end{equation*}
The strong type $(p,q)$ inequalities then follow by Marcinkiewicz interpolation theorem.

In 1974, Muckenhoupt and Wheeden studied the weighted boundedness of $I_{\alpha}$ when $0<\alpha<n$. A weight $w$ is a nonnegative, locally integrable function on $\mathbb R^n$ that take values in $(0,\infty)$ almost everywhere. Given a weight $w$ and a measurable set $E\subset\mathbb R^n$, we use the notation
\begin{equation*}
w(E):=\int_{E}w(x)\,dx
\end{equation*}
to denote the $w$-measure of the set $E$. The complement of the set $E$ is denoted by $E^{\complement}$, and the characteristic function of the set $E$ is denoted by $\chi_{E}$: $\chi_{E}(x)=1$ if $x\in E$ and 0 if $x\notin E$. For given $x_0\in\mathbb R^n$ and $r>0$, $B(x_0,r)$ denotes the open ball centered at $x_0$ with radius $r$. Given a ball $B$ and $\lambda>0$, $\lambda B$ denotes the ball with the same center as $B$ whose radius is $\lambda$ times that of $B$. Following \cite{muckenhoupt2}, let us give the definitions of some weight classes.
\begin{defin}[\cite{muckenhoupt2}]
A weight function $w$ is said to belong to the Muckenhoupt--Wheeden class $A(p,q)$ for $1<p<q<\infty$, if there exists a constant $C>0$ such that
\begin{equation*}
\bigg(\frac{1}{|B|}\int_{B}w(x)^q\,dx\bigg)^{1/q}\bigg(\frac{1}{|B|}\int_{B}w(x)^{-p'}\,dx\bigg)^{1/{p'}}\leq C<\infty,
\end{equation*}
for every ball $B$ in $\mathbb R^n$, where $p'$ denotes the conjugate exponent of $p>1$ such that $1/{p'}+1/p=1$, and $p'=1$ when $p=\infty$. When $p=1$, a weight function $w$ is in the Muckenhoupt--Wheeden class $A(1,q)$ with $1<q<\infty$, if there exists a constant $C>0$ such that
\begin{equation*}
\bigg(\frac{1}{|B|}\int_{B}w(x)^q\,dx\bigg)^{1/q}\bigg(\underset{x\in B}{\mathrm{ess\,sup}}\;\frac{1}{w(x)}\bigg)\leq C<\infty,
\end{equation*}
for every ball $B$ in $\mathbb R^n$. A weight function $w$ is said to belong to the Muckenhoupt--Wheeden class $A(p,\infty)$ with $1<p<\infty$, if there exists a constant $C>0$ such that
\begin{equation*}
\bigg(\underset{x\in B}{\mathrm{ess\,sup}}\;w(x)\bigg)\bigg(\frac{1}{|B|}\int_{B}w(x)^{-p'}\,dx\bigg)^{1/{p'}}\leq C<\infty,
\end{equation*}
for every ball $B$ in $\mathbb R^n$.
\end{defin}
The weighted Lebesgue spaces and weak spaces with respect to the measure $w(x)\,dx$ are denoted by $L^p(w)$ and $WL^p(w)$ with $1\leq p<\infty$, respectively. In 1974, Muckenhoupt and Wheeden proved the following strong type and weak type estimates of $M_{\alpha}$ and $I_{\alpha}$ (see \cite[Theorems 2 through 5]{muckenhoupt2}).
\newtheorem*{thma}{Theorem A}
\begin{thma}[\cite{muckenhoupt2}]
Let $0<\alpha<n$, $1<p<n/{\alpha}$ and $1/q=1/p-{\alpha}/n$. If $w\in A(p,q)$, then the operators $M_{\alpha}$ and $I_{\alpha}$ are bounded from $L^p(w^p)$ to $L^q(w^q)$. If $p=1$ and $w\in A(1,q)$ with $q=n/{(n-\alpha)}$, then the operators $M_{\alpha}$ and $I_{\alpha}$ are bounded from $L^1(w)$ to $WL^q(w^q)$.
\end{thma}

\begin{defin}[\cite{john,muckenhoupt2}]
A locally integrable function $\hbar$ is said to be in the bounded mean oscillation space $\mathrm{BMO}(\mathbb R^n)$, if
\begin{equation*}
\|\hbar\|_*:=\sup_{B}\frac{1}{|B|}\int_B|\hbar(x)-\hbar_B|\,dx<\infty,
\end{equation*}
where $\hbar_B$ stands for the average of $\hbar$ on $B$; i.e., $\hbar_B:=\frac{1}{|B|}\int_B \hbar(y)\,dy$ and the supremum is taken over all balls $B$ in $\mathbb R^n$. For a weight function $w$ on $\mathbb R^n$, the weighted version of $\mathrm{BMO}$ is denoted by $\mathrm{BMO}(w)$. We say that $\hbar$ is in the weighted space $\mathrm{BMO}(w)$, if there is a constant $C>0$ such that for any ball $B$ in $\mathbb R^n$,
\begin{equation*}
\bigg(\underset{x\in B}{\mathrm{ess\,sup}}\;w(x)\bigg)\bigg(\frac{1}{|B|}\int_B|\hbar(x)-\hbar_B|\,dx\bigg)\leq C<\infty.
\end{equation*}
\end{defin}
Moreover, for the critical index $p=n/{\alpha}$, Muckenhoupt and Wheeden proved the following result (see \cite[Theorem 7]{muckenhoupt2}).
\newtheorem*{thmb}{Theorem B}
\begin{thmb}[\cite{muckenhoupt2}]
Let $0<\alpha<n$ and $p=n/{\alpha}$. If $w\in A(p,\infty)$, then the Riesz potential operator $I_{\alpha}$ is bounded from $L^p(w^p)$ to $\mathrm{BMO}(w)$.
\end{thmb}
We now introduce the function space $L^{\infty}(w)$, the weighted version of $L^{\infty}(\mathbb R^n)$. We say that a measurable function $\hslash$ belongs to the weighted space $L^{\infty}(w)$, if there is a constant $C>0$ such that for any ball $B$ in $\mathbb R^n$,
\begin{equation*}
\bigg(\underset{x\in B}{\mathrm{ess\,sup}}\;w(x)\bigg)\Big(\big\|\hslash\big\|_{L^{\infty}(B)}\Big)\leq C<\infty,
\end{equation*}
where
\begin{equation*}
\big\|\hslash\big\|_{L^{\infty}(B)}:=\underset{x\in B}{\mbox{ess\,sup}}\;|\hslash(x)|=\inf\big\{M>0:\big|\{x\in E:|\hslash(x)|>M\}\big|=0\big\}.
\end{equation*}
We prove in this paper that the strong type $(p,\infty)$ result remains valid for $M_{\alpha}$ in our weighted spaces, under the assumptions that $w\in A(p,\infty)$ and $p=n/{\alpha}$. Actually, by H\"{o}lder's inequality and the condition $w\in A(p,\infty)$, we get
\begin{equation*}
\begin{split}
&\bigg(\underset{y\in B(x,r)}{\mathrm{ess\,sup}}\;w(y)\bigg)\cdot\frac{1}{|B(x,r)|^{1-\alpha/n}}\int_{|x-y|<r}|f(y)|\,dy\\
&\leq\bigg(\int_{B(x,r)}|f(y)|^pw(y)^{p}\,dy\bigg)^{1/p}
\bigg(\underset{y\in B(x,r)}{\mathrm{ess\,sup}}\;w(y)\bigg)\bigg(\frac{1}{|B(x,r)|}\int_{B(x,r)}w(y)^{-p'}\,dy\bigg)^{1/{p'}}\\
&\leq C\|f\|_{L^p(w^p)}.
\end{split}
\end{equation*}
Thus, the fractional maximal operator $M_{\alpha}$ is bounded from $L^p(w^p)$ to $L^{\infty}(w)$.

Let $\mathbf{S}^{n-1}=\{x\in\mathbb R^n:|x|=1\}$ be the unit sphere in $\mathbb R^n$ ($n\geq 2$) equipped with the normalized Lebesgue measure $d\sigma(x')$. Let $\alpha\in(0,n)$ and $\Omega$ be a homogeneous function of degree zero on $\mathbb R^n$ satisfying $\Omega\in L^{s}(\mathbf{S}^{n-1})$ with $s\geq1$.  We study the homogeneous fractional integral operator $T_{\Omega,\alpha}$, which is defined by
\begin{equation}
T_{\Omega,\alpha}f(x):=\int_{\mathbb R^n}\frac{\Omega(x-y)}{|x-y|^{n-\alpha}}f(y)\,dy,\quad x\in\mathbb R^n.
\end{equation}
Given any $\alpha\in(0,n)$ and $f\in L^1_{\mathrm{loc}}(\mathbb R^n)$, the fractional maximal function of $f$ with homogeneous kernel is defined by
\begin{equation}
M_{\Omega,\alpha}f(x):=\sup_{r>0}\frac{1}{|B(x,r)|^{1-\alpha/n}}\int_{|x-y|<r}|\Omega(x-y)||f(y)|\,dy,\quad x\in\mathbb R^n,
\end{equation}
where the supremum on the right-hand side is taken over all $r>0$.
\begin{description}
  \item[1] When $\alpha=0$ and $\Omega\equiv1$, $M_{\Omega,\alpha}$ is the standard Hardy--Littlewood maximal operator on $\mathbb R^n$, which is denoted by $M$;
  \item[2] When $\Omega\equiv1$, $M_{\Omega,\alpha}$($T_{\Omega,\alpha}$) is just the fractional maximal (integral) operator $M_{\alpha}$($I_{\alpha}$) on $\mathbb R^n$.
\end{description}
In 1971, Muckenhoupt and Wheeden \cite{muckenhoupt1} studied weighted norm inequalities for $T_{\Omega,\alpha}$ with the weight $w(x)=|x|^\beta$. Some weak type estimates with power weights for $M_{\Omega,\alpha}$ and $T_{\Omega,\alpha}$ were obtained by Ding in \cite{ding3}. In 1998, Ding and Lu considered weighted norm inequalities for $M_{\Omega,\alpha}$ and $T_{\Omega,\alpha}$ with more general weights. More precisely, they proved the $(L^p(w^p),L^q(w^q))$ boundedness of the fractional maximal operator $M_{\Omega,\alpha}$ and the fractional integral operator $T_{\Omega,\alpha}$ (see \cite[Theorem 1 and Proposition 1]{ding1} and \cite[Chapter 3]{lu}).
\newtheorem*{thmc}{Theorem C}
\begin{thmc}[\cite{ding1}]
Let $0<\alpha<n$, $1\le s'<p<n/{\alpha}$ and $1/q=1/p-{\alpha}/n$. If $\Omega\in L^s(\mathbf{S}^{n-1})$ and $w^{s'}\in A(p/{s'},q/{s'})$, then the operators $M_{\Omega,\alpha}$ and $T_{\Omega,\alpha}$ are all bounded from $L^p(w^p)$ to $L^q(w^q)$.
\end{thmc}

\begin{defin}[\cite{ding2}]
We say that $\Omega$ satisfies the $L^s$-Dini smoothness condition $\mathfrak{D}_{s}$, if $\Omega\in L^s(\mathbf{S}^{n-1})$ is homogeneous of degree zero on $\mathbb R^n$ with $1\leq s<\infty$, and
\begin{equation*}
\int_0^1\omega_s(\delta)\frac{d\delta}{\delta}<\infty,
\end{equation*}
where $\omega_s(\delta)$ denotes the integral modulus of continuity of order $s$, which is defined by
\begin{equation*}
\omega_s(\delta):=\sup_{|\rho|<\delta}\bigg(\int_{\mathbf{S}^{n-1}}\big|\Omega(\rho x')-\Omega(x')\big|^sd\sigma(x')\bigg)^{1/s}
\end{equation*}
and $\rho$ is a rotation in $\mathbb R^n$ and $|\rho|:=\|\rho-I\|=\sup_{x'\in\mathbf{S}^{n-1}}\big|\rho x'-x'\big|$. We also say that $\Omega$ satisfies the $L^{\infty}$-Dini smoothness condition $\mathfrak{D}_{\infty}$, if $\Omega\in L^{\infty}(\mathbf{S}^{n-1})$ is homogeneous of degree zero on $\mathbb R^n$, and
\begin{equation*}
\int_0^1\omega_{\infty}(\delta)\frac{d\delta}{\delta}<\infty,
\end{equation*}
where $\omega_{\infty}(\delta)$ is defined by
\begin{equation*}
\omega_{\infty}(\delta):=\sup_{|\rho|<\delta,x'\in \mathbf{S}^{n-1}}\big|\Omega(\rho x')-\Omega(x')\big|,
\end{equation*}
where $x'=x/{|x|}$ for any $\mathbb R^n\ni x\neq0$.
\end{defin}

In 2002, Ding and Lu considered the critical case $p=n/{\alpha}$, and proved the weighted boundedness of $T_{\Omega,\alpha}$ from $L^p(w^p)$ to $\mathrm{BMO}(w)$ (see \cite[Theorem 1]{ding2}).
\newtheorem*{thmd}{Theorem D}
\begin{thmd}[\cite{ding2}]
Let $0<\alpha<n$ and $p=n/{\alpha}$. If $\Omega$ satisfies the $L^s$-Dini condition $\mathfrak{D}_{s}$ with $s>n/{(n-\alpha)}$ and $w^{s'}\in A(p/{s'},\infty)$, then the homogeneous fractional integral operator $T_{\Omega,\alpha}$ is bounded from $L^p(w^p)$ to $\mathrm{BMO}(w)$.
\end{thmd}
Let $L^{\infty}(w)$ denote the space of all essentially bounded measurable functions with respect to $w$ on $\mathbb R^n$. In this paper, we are able to prove that the fractional maximal operator $M_{\Omega,\alpha}$ is bounded from $L^p(w^p)$ to $L^{\infty}(w)$ when $0<\alpha<n$ and $\Omega\in L^s(\mathbf{S}^{n-1})$ with $s>n/{(n-\alpha)}$.

In 2009, Komori and Shirai \cite{komori} defined and investigated the weighted Morrey spaces $L^{p,\kappa}(\mu,\nu)$ which could be viewed as an extension of weighted Lebesgue spaces, and obtained the boundedness of some classical integral operators on these weighted spaces. Let us now give the definition of weighted Morrey spaces.

\begin{defin}[\cite{komori}]
Let $1\le p<\infty$ and $0<\kappa<1$. For two weights $\mu$ and $\nu$ on $\mathbb R^n$, the weighted Morrey space $L^{p,\kappa}(\mu,\nu)$ is defined by
\begin{equation*}
L^{p,\kappa}(\mu,\nu):=\Big\{f\in L^p_{\mathrm{loc}}(\mu):\|f\|_{L^{p,\kappa}(\mu,\nu)}<\infty\Big\},
\end{equation*}
where
\begin{equation*}
\|f\|_{L^{p,\kappa}(\mu,\nu)}:=\sup_{B}\bigg(\frac{1}{\nu(B)^{\kappa}}\int_B|f(x)|^p\mu(x)\,dx\bigg)^{1/p}
\end{equation*}
and the supremum is taken over all balls $B$ in $\mathbb R^n$. If $\mu=\nu$, then $L^{p,\kappa}(\mu,\nu)$ is simply denoted by $L^{p,\kappa}(\mu)$.
\end{defin}
It is known that $L^{p,\kappa}$ is an extension of $L^p$ in the sense that $L^{p,0}=L^p$. Komori and Shirai obtained the strong type estimates for $M_{\alpha}$ and $I_{\alpha}$ in weighted Morrey spaces based on Theorem A (see \cite[Theorems 3.5 and 3.6]{komori}).

\newtheorem*{thme}{Theorem E}
\begin{thme}[\cite{komori}]
If $0<\alpha<n$, $1<p<n/{\alpha}$, $1/q=1/p-\alpha/n$, $0<\kappa<p/q$ and $w\in A(p,q)$, then the fractional maximal operator $M_\alpha$ and the Riesz potential operator $I_{\alpha}$ are bounded from $L^{p,\kappa}(w^p,w^q)$ to $L^{q,{\kappa q}/p}(w^q)$.
\end{thme}

In 2013, we established the boundedness properties of $M_{\Omega,\alpha}$ and $T_{\Omega,\alpha}$ in weighted Morrey spaces based on Theorem C (see \cite[Theorems 1.1 and 1.2]{wang1}).

\newtheorem*{thmf}{Theorem F}
\begin{thmf}[\cite{wang1}]
Suppose that $\Omega\in L^s(\mathbf{S}^{n-1})$ with $1<s\le\infty$. If $0<\alpha<n$, $1\le s'<p<n/{\alpha}$, $1/q=1/p-{\alpha}/n$, $0<\kappa<p/q$ and $w^{s'}\in A(p/{s'},q/{s'})$, then the fractional maximal operator $M_{\Omega,\alpha}$ and the fractional integral operator $T_{\Omega,\alpha}$ are bounded from $L^{p,\kappa}(w^p,w^q)$ to $L^{q,{\kappa q}/p}(w^q)$.
\end{thmf}

In the present work, we study the homogeneous fractional integral operators and fractional maximal operators on the weighted Morrey spaces $L^{p,\kappa}(w^p,w^q)$. Motivated by Theorems D and F, it is natural to ask whether there is a corresponding estimate for the critical case $\kappa=p/q=1-{(\alpha p)}/n$. We give a positive answer to this question. It can be shown that the homogeneous fractional integral operator $T_{\Omega,\alpha}$ is bounded from $L^{p,\kappa}(w^p,w^q)$ to $\mathrm{BMO}(\mathbb R^n)$, if $\Omega$ satisfies the Dini smoothness condition $\mathfrak{D}_{s}$ with $1<s\le\infty$. Moreover, the fractional maximal operator $M_{\Omega,\alpha}$ is shown to be bounded from $L^{p,\kappa}(w^p,w^q)$ to $L^{\infty}(\mathbb R^n)$, if $\Omega\in L^s(\mathbf{S}^{n-1})$ with $1<s\le\infty$.

Throughout this paper, $C$ is a positive constant which is independent of the main parameters and not necessarily the same at each occurrence.

\section{Main results}
Now let us formulate our main results as follows.
\begin{thm}\label{thm1}
Let $0<\alpha<n$ and $p=n/{\alpha}$. If $\Omega\in L^s(\mathbf{S}^{n-1})$ with $s>n/{(n-\alpha)}$ and $w^{s'}\in A(p/{s'},\infty)$, then the
fractional maximal operator $M_{\Omega,\alpha}$ is bounded from $L^p(w^p)$ to $L^{\infty}(w)$.
\end{thm}

\begin{thm}\label{thm2}
Suppose that $\Omega$ satisfies the $L^s$-Dini condition $\mathfrak{D}_{s}$ with $1<s\le\infty$. If $0<\alpha<n$, $s'\leq p<n/{\alpha}$, $1/q=1/p-{\alpha}/n$, $\kappa=p/q=1-{(\alpha p)}/n$ and $w^{s'}\in A(p/{s'},q/{s'})$, then the homogeneous fractional integral operator $T_{\Omega,\alpha}$ is bounded from $L^{p,\kappa}(w^p,w^q)$ to $\mathrm{BMO}(\mathbb R^n)$.
\end{thm}

In particular, when $\Omega\equiv1$ and $s=\infty$, we have the following result.

\begin{cor}
If $0<\alpha<n$, $1\leq p<n/{\alpha}$, $1/q=1/p-{\alpha}/n$, $\kappa=p/q=1-{(\alpha p)}/n$ and $w\in A(p,q)$, then the Riesz potential operator $I_{\alpha}$ is bounded from $L^{p,\kappa}(w^p,w^q)$ to $\mathrm{BMO}(\mathbb R^n)$.
\end{cor}

\begin{rem}
As a special case $w\equiv1$, this type of estimate was shown by Adams and implicit in \cite{adams} (see also \cite{adams1,adams2}).
\end{rem}

\section{Proofs of Theorems \ref{thm1} and \ref{thm2}}
In this section, we will prove the conclusions of Theorems \ref{thm1} and \ref{thm2}.
\begin{proof}[Proof of Theorem $\ref{thm1}$]
Given a ball $B(x,r)$ with center $x\in\mathbb R^n$ and radius $r\in(0,\infty)$, by using H\"{o}lder's inequality twice (since $p>s'$), we get
\begin{equation*}
\begin{split}
&\bigg(\underset{y\in B(x,r)}{\mathrm{ess\,sup}}\;w(y)\bigg)\cdot
\frac{1}{|B(x,r)|^{1-\alpha/n}}\int_{|x-y|<r}|\Omega(x-y)||f(y)|\,dy\\
&\leq\bigg(\underset{y\in B(x,r)}{\mathrm{ess\,sup}}\;w(y)\bigg)\cdot
\frac{1}{|B(x,r)|^{1/{p'}}}\bigg(\int_{|x-y|<r}|\Omega(x-y)|^sdy\bigg)^{1/s}\bigg(\int_{|x-y|<r}|f(y)|^{s'}dy\bigg)^{1/{s'}}\\
&\leq\bigg(\underset{y\in B(x,r)}{\mathrm{ess\,sup}}\;w(y)\bigg)\cdot
\frac{1}{|B(x,r)|^{1/{p'}}}\bigg(\int_{|x-y|<r}|\Omega(x-y)|^sdy\bigg)^{1/s}\\
&\times\bigg(\int_{B(x,r)}|f(y)|^{p}w(y)^{p}dy\bigg)^{1/p}
\bigg(\int_{B(x,r)}w^{s'}(y)^{-\big(\frac{p}{s'}\big)'}dy\bigg)^{\frac{1}{s'\big(\frac{p}{s'}\big)'}}.
\end{split}
\end{equation*}
Note that $\Omega\in L^s(\mathbf{S}^{n-1})$. Using polar coordinates, we obtain
\begin{equation}\label{omega}
\begin{split}
&\bigg(\int_{|x-y|<r}|\Omega(x-y)|^sdy\bigg)^{1/s}\\
&=\bigg(\int_{0}^{r}\int_{\mathbf{S}^{n-1}}|\Omega(x')|^s\varrho^{n-1}d\sigma(x')d\varrho\bigg)^{1/s}\\
&=C\cdot r^{n/s}\|\Omega\|_{L^s(\mathbf{S}^{n-1})}=C|B(x,r)|^{1/s}\|\Omega\|_{L^s(\mathbf{S}^{n-1})}.
\end{split}
\end{equation}
On the other hand, a direct computation shows that
\begin{equation*}
\begin{split}
\frac{1}{s'\big(\frac{p}{s'}\big)'}=\frac{1}{s'}\cdot\frac{p-s'}{p}=\frac{1}{s'}-\frac{1}{p}=\frac{1}{p'}-\frac{1}{s}.
\end{split}
\end{equation*}
This, combined with \eqref{omega}, yields
\begin{equation*}
\begin{split}
&\bigg(\underset{y\in B(x,r)}{\mathrm{ess\,sup}}\;w(y)\bigg)\cdot
\frac{1}{|B(x,r)|^{1-\alpha/n}}\int_{|x-y|<r}|\Omega(x-y)||f(y)|\,dy\\
&\leq C\|\Omega\|_{L^s(\mathbf{S}^{n-1})}\bigg(\int_{B(x,r)}|f(y)|^pw(y)^{p}\,dy\bigg)^{1/p}\\
&\times\bigg(\underset{y\in B(x,r)}{\mathrm{ess\,sup}}\;w(y)\bigg)
\bigg(\frac{1}{|B(x,r)|}\int_{B(x,r)}w^{s'}(y)^{-\big(\frac{p}{s'}\big)'}dy\bigg)^{\frac{1}{s'\big(\frac{p}{s'}\big)'}}\\
&=C\|\Omega\|_{L^s(\mathbf{S}^{n-1})}\bigg(\int_{B(x,r)}|f(y)|^pw(y)^{p}\,dy\bigg)^{1/p}\\
&\times\bigg[\bigg(\underset{y\in B(x,r)}{\mathrm{ess\,sup}}\;w^{s'}(y)\bigg)
\bigg(\frac{1}{|B(x,r)|}\int_{B(x,r)}w^{s'}(y)^{-\big(\frac{p}{s'}\big)'}dy\bigg)^{\frac{1}{\big(\frac{p}{s'}\big)'}}\bigg]^{\frac{1}{s'}}\\
&\leq C\|f\|_{L^p(w^p)},
\end{split}
\end{equation*}
where in the last step we have used the fact that $w^{s'}\in A(p/{s'},\infty)$. This ends the proof of Theorem \ref{thm1} by the definition of $L^{\infty}(w)$.
\end{proof}

Before we proceed to prove our next theorem, we need the following estimate which can be found in \cite[Lemma 1]{ding2}.
\begin{lem}[\cite{ding2}]\label{lem1}
Suppose that $0<\alpha<n$ and $\Omega$ satisfies the $L^s$-Dini condition $\mathfrak{D}_{s}$ with $1<s\le\infty$. If there exists a constant $0<\tau<1/2$ such that if $|x|<\tau R$, then we have
\begin{equation*}
\bigg(\int_{R\leq|z|<2R}\bigg|\frac{\Omega(z-x)}{|z-x|^{n-\alpha}}-\frac{\Omega(z)}{|z|^{n-\alpha}}\bigg|^sdz\bigg)^{1/s}
\leq C\cdot R^{n/s-(n-\alpha)}\bigg(\frac{|x|}{R}+\int_{\frac{|x|}{2R}}^{\frac{|x|}{R}}\omega_s(\delta)\frac{d\delta}{\delta}\bigg),
\end{equation*}
where the constant $C>0$ is independent of $R$ and $x$.
\end{lem}
\begin{rem}
Following the same arguments as in the proof of Lemma 5 in \cite{kurtz}, we can also establish the above lemma on the homogeneous kernel $\Omega(z)$.
\end{rem}
\begin{proof}[Proof of Theorem $\ref{thm2}$]
For any ball $B=B(x_0,r)$ with $(x_0,r)\in\mathbb R^n\times(0,\infty)$ and $f\in L^{p,\kappa}(w^p,w^q)$, we are going to give the estimation of the following expression
\begin{equation}\label{mainw}
\frac{1}{|B|}\int_{B}\big|T_{\Omega,\alpha}f(x)-(T_{\Omega,\alpha}f)_{B}\big|\,dx.
\end{equation}
For this purpose, we decompose $f$ as $f=f_1+f_2$, where $f_1=f\cdot\chi_{4B}$, $f_2=f\cdot\chi_{(4B)^{\complement}}$, $4B=B(x_0,4r)$. By the linearity of the fractional integral operator $T_{\Omega,\alpha}$, the expression \eqref{mainw} will be divided into two parts. That is,
\begin{equation*}
\begin{split}
&\frac{1}{|B|}\int_{B}\big|T_{\Omega,\alpha}f(x)-(T_{\Omega,\alpha}f)_{B}\big|\,dx\\
&\leq \frac{1}{|B|}\int_{B}\big|T_{\Omega,\alpha}f_1(x)-(T_{\Omega,\alpha}f_1)_{B}\big|\,dx
+\frac{1}{|B|}\int_{B}\big|T_{\Omega,\alpha}f_2(x)-(T_{\Omega,\alpha}f_2)_{B}\big|\,dx\\
&:=\mathrm{I+II}.
\end{split}
\end{equation*}
Let us first consider the term $\mathrm{I}$. When $x\in B$ and $y\in 4B$, we have $|x-y|<5r$. An application of Fubini's theorem yields
\begin{equation*}
\begin{split}
\mathrm{I}&\leq\frac{2}{|B|}\int_{B}|T_{\Omega,\alpha}f_1(x)|\,dx\\
&\leq\frac{2}{|B|}\int_{B}\bigg(\int_{4B}\frac{|\Omega(x-y)|}{|x-y|^{n-\alpha}}|f(y)|\,dy\bigg)dx\\
&\leq\frac{2}{|B|}\int_{4B}|f(y)|\bigg(\int_{|x-y|<5r}\frac{|\Omega(x-y)|}{|x-y|^{n-\alpha}}\,dx\bigg)dy.
\end{split}
\end{equation*}
Since $\Omega\in L^s(\mathbf{S}^{n-1})$, we have $\Omega\in L^1(\mathbf{S}^{n-1})$, and
\begin{equation*}
\|\Omega\|_{L^1(\mathbf{S}^{n-1})}\leq C\|\Omega\|_{L^s(\mathbf{S}^{n-1})}.
\end{equation*}
Note that $\alpha>0$. Using polar coordinates, we can deduce that
\begin{equation}\label{wang1}
\begin{split}
\int_{|x-y|<5r}\frac{|\Omega(x-y)|}{|x-y|^{n-\alpha}}\,dx
&=\int_{0}^{5r}\int_{\mathbf{S}^{n-1}}\frac{|\Omega(x')|}{\varrho^{n-\alpha}}\varrho^{n-1}d\sigma(x')d\varrho\\
&=C\cdot r^{\alpha}\|\Omega\|_{L^1(\mathbf{S}^{n-1})}\leq C|B|^{\alpha/n}\|\Omega\|_{L^s(\mathbf{S}^{n-1})}.
\end{split}
\end{equation}
Applying the H\"older inequality, we obtain
\begin{equation}\label{wang2}
\begin{split}
\int_{4B}|f(y)|\,dy
&\leq\bigg(\int_{4B}|f(y)|^pw(y)^p\,dy\bigg)^{1/p}\bigg(\int_{4B} w(y)^{-p'}dy\bigg)^{1/{p'}}.
\end{split}
\end{equation}
Notice that
\begin{equation}\label{cal3}
s'\Big(\frac{p}{s'}\Big)'=\frac{s'}{1-\frac{s'}{p}}=\frac{ps'}{p-s'}.
\end{equation}
By a simple calculation, we can see that when $s'\geq1$,
\begin{equation*}
\frac{ps'}{p-s'}\geq\frac{p}{p-1}=p'.
\end{equation*}
Using H\"older's inequality and the condition $w^{s'}\in A(p/{s'},q/{s'})$, we get
\begin{equation}\label{wang3}
\begin{split}
\bigg(\frac{1}{|4B|}\int_{4B} w(y)^{-p'}dy\bigg)^{1/{p'}}
&\leq\bigg(\frac{1}{|4B|}\int_{4B} w(y)^{-s'\big(\frac{p}{s'}\big)'}dy\bigg)^{\frac{1}{s'\big(\frac{p}{s'}\big)'}}\\
&=\bigg(\frac{1}{|4B|}\int_{4B}w^{s'}(y)^{-\big(\frac{p}{s'}\big)'}dy\bigg)^{\frac{1}{s'\big(\frac{p}{s'}\big)'}}\\
&\leq C\bigg(\frac{1}{|4B|}\int_{4B}w^{s'}(y)^{\frac{q}{s'}}dy\bigg)^{-\frac{s'}{q}\cdot\frac{1}{s'}}\\
&=C\bigg(\frac{1}{|4B|}\int_{4B}w(y)^{q}dy\bigg)^{-1/{q}}.
\end{split}
\end{equation}
A direct calculation shows that
\begin{equation}\label{cal}
\frac{1}{p'}+\frac{1}{q}=1-\frac{1}{p}+\frac{1}{q}=1-\frac{\alpha}{n}.
\end{equation}
Hence, from \eqref{wang1}, \eqref{wang2} and \eqref{wang3}, it follows that
\begin{equation*}
\begin{split}
\mathrm{I}&\leq C\cdot\frac{1}{|B|^{1-\alpha/n}}\int_{4B}|f(y)|\,dy\\
&\leq C\cdot\frac{1}{|B|^{1-\alpha/n}}\cdot|4B|^{\frac{1}{p'}+\frac{1}{q}}
\bigg(\int_{4B}|f(y)|^pw(y)^pdy\bigg)^{1/p}\bigg(\int_{4B}w(y)^{q}dy\bigg)^{-1/{q}}\\
&\leq C\cdot\frac{w^q(4B)^{\kappa/p}}{w^q(4B)^{1/q}}\|f\|_{L^{p,\kappa}(w^p,w^q)}=C\|f\|_{L^{p,\kappa}(w^p,w^q)},
\end{split}
\end{equation*}
where the last equality holds because $\kappa=p/q$. Now we turn our attention to the second term $\mathrm{II}$. It is easy to see that
\begin{equation*}
\begin{split}
\mathrm{II}&=\frac{1}{|B|}\int_{B}\big|T_{\Omega,\alpha}f_2(x)-(T_{\Omega,\alpha}f_2)_{B}\big|\,dx\\
&\leq\frac{1}{|B|}\int_{B}\bigg\{\frac{1}{|B|}\int_{B}\big|T_{\Omega,\alpha}f_2(x)-T_{\Omega,\alpha}f_2(y)\big|\,dy\bigg\}dx\\
&\leq\frac{1}{|B|}\int_{B}\bigg\{\frac{1}{|B|}\int_{B}\bigg(\sum_{j=2}^\infty\int_{2^{j+1}B\backslash 2^jB}
\Big|\frac{\Omega(x-z)}{|x-z|^{n-\alpha}}-\frac{\Omega(y-z)}{|y-z|^{n-\alpha}}\Big||f(z)|dz\bigg)dy\bigg\}dx.
\end{split}
\end{equation*}
Since
\begin{equation*}
\begin{split}
\Big|\frac{\Omega(x-z)}{|x-z|^{n-\alpha}}-\frac{\Omega(y-z)}{|y-z|^{n-\alpha}}\Big|
&\leq\Big|\frac{\Omega(x-z)}{|x-z|^{n-\alpha}}-\frac{\Omega(x_0-z)}{|x_0-z|^{n-\alpha}}\Big|\\
&+\Big|\frac{\Omega(x_0-z)}{|x_0-z|^{n-\alpha}}-\frac{\Omega(y-z)}{|y-z|^{n-\alpha}}\Big|,
\end{split}
\end{equation*}
then by using H\"{o}lder's inequality and Minkowski's inequality, we have
\begin{equation*}
\begin{split}
&\int_{2^{j+1}B\backslash 2^jB}
\Big|\frac{\Omega(x-z)}{|x-z|^{n-\alpha}}-\frac{\Omega(y-z)}{|y-z|^{n-\alpha}}\Big||f(z)|dz\\
&\leq\bigg(\int_{2^{j+1}B\backslash 2^jB}\Big|\frac{\Omega(x-z)}{|x-z|^{n-\alpha}}-\frac{\Omega(y-z)}{|y-z|^{n-\alpha}}\Big|^sdz\bigg)^{1/s}
\bigg(\int_{2^{j+1}B\backslash 2^jB}|f(z)|^{s'}dz\bigg)^{1/{s'}}\\
&\leq\bigg(\int_{2^{j+1}B\backslash 2^jB}\Big|\frac{\Omega(x-z)}{|x-z|^{n-\alpha}}-\frac{\Omega(x_0-z)}{|x_0-z|^{n-\alpha}}\Big|^sdz\bigg)^{1/s}
\bigg(\int_{2^{j+1}B\backslash 2^jB}|f(z)|^{s'}dz\bigg)^{1/{s'}}\\
&+\bigg(\int_{2^{j+1}B\backslash 2^jB}\Big|\frac{\Omega(x_0-z)}{|x_0-z|^{n-\alpha}}-\frac{\Omega(y-z)}{|y-z|^{n-\alpha}}\Big|^sdz\bigg)^{1/s}
\bigg(\int_{2^{j+1}B\backslash 2^jB}|f(z)|^{s'}dz\bigg)^{1/{s'}}.
\end{split}
\end{equation*}
For any given $x\in B$ and $z\in 2^{j+1}B\backslash 2^jB$ with $j\geq2$, a trivial computation shows that
\begin{equation*}
|x-x_0|<\frac{\,1\,}{4}|z-x_0|.
\end{equation*}
According to Lemma \ref{lem1}(if we take $R=2^jr$), we get
\begin{equation*}
\begin{split}
&\bigg(\int_{2^{j+1}B\backslash 2^jB}\Big|\frac{\Omega(x-z)}{|x-z|^{n-\alpha}}-\frac{\Omega(x_0-z)}{|x_0-z|^{n-\alpha}}\Big|^sdz\bigg)^{1/s}\\
&=\bigg(\int_{2^jr\leq|z-x_0|<2^{j+1}r}\Big|\frac{\Omega(z-x_0-(x-x_0))}{|z-x_0-(x-x_0)|^{n-\alpha}}
-\frac{\Omega(z-x_0)}{|z-x_0|^{n-\alpha}}\Big|^sdz\bigg)^{1/s}\\
&\leq C(2^jr)^{n/s-(n-\alpha)}
\bigg[\frac{|x-x_0|}{2^jr}+\int_{\frac{|x-x_0|}{2^{j+1}r}}^{\frac{|x-x_0|}{2^jr}}\omega_s(\delta)\frac{d\delta}{\delta}\bigg]\\
&\leq C|2^{j+1}B|^{1/s-(1-\alpha/n)}
\bigg[\frac{1}{2^j}+\int_{\frac{|x-x_0|}{2^{j+1}r}}^{\frac{|x-x_0|}{2^jr}}\omega_s(\delta)\frac{d\delta}{\delta}\bigg].
\end{split}
\end{equation*}
Similarly, for any given $y\in B$, we can also obtain
\begin{equation*}
\begin{split}
&\bigg(\int_{2^{j+1}B\backslash 2^jB}\Big|\frac{\Omega(x_0-z)}{|x_0-z|^{n-\alpha}}-\frac{\Omega(y-z)}{|y-z|^{n-\alpha}}\Big|^sdz\bigg)^{1/s}\\
&\leq C(2^jr)^{n/s-(n-\alpha)}
\bigg[\frac{|y-x_0|}{2^jr}+\int_{\frac{|y-x_0|}{2^{j+1}r}}^{\frac{|y-x_0|}{2^jr}}\omega_s(\delta)\frac{d\delta}{\delta}\bigg]\\
&\leq C|2^{j+1}B|^{1/s-(1-\alpha/n)}
\bigg[\frac{1}{2^j}+\int_{\frac{|y-x_0|}{2^{j+1}r}}^{\frac{|y-x_0|}{2^jr}}\omega_s(\delta)\frac{d\delta}{\delta}\bigg].
\end{split}
\end{equation*}
On the other hand, we now claim that the following inequality holds:
\begin{align}\label{wang83}
&\bigg(\int_{2^{j+1}B\backslash 2^jB}|f(z)|^{s'}dz\bigg)^{1/{s'}}\\
&\leq C\bigg(\int_{2^{j+1}B}|f(z)|^{p}w(z)^{p}dz\bigg)^{1/p}
\bigg(\int_{2^{j+1}B}w(z)^{q}dz\bigg)^{-1/{q}}\times|2^{j+1}B|^{(1-\alpha/n)-1/s}\notag.
\end{align}
To prove \eqref{wang83}, we consider the following two cases: $p=s'$ and $p>s'$.

\textbf{Case 1:} $p=s'$. In this case, it follows from the condition $w^{p}\in A(1,q/{p})$ and \eqref{cal} that
\begin{equation*}
\begin{split}
&\bigg(\int_{2^{j+1}B\backslash 2^jB}|f(z)|^{s'}dz\bigg)^{1/{s'}}\\
&\leq\bigg(\int_{2^{j+1}B}|f(z)|^{p}w(z)^{p}w(z)^{-p}dz\bigg)^{1/{p}}\\
&\leq \bigg(\int_{2^{j+1}B}|f(z)|^{p}w(z)^{p}dz\bigg)^{1/p}\bigg(\underset{z\in 2^{j+1}B}{\mbox{ess\,sup}}\;\frac{1}{w^p(z)}\bigg)^{1/{p}}\\
&\leq C\bigg(\int_{2^{j+1}B}|f(z)|^{p}w(z)^{p}dz\bigg)^{1/p}
\bigg(\int_{2^{j+1}B}w(z)^{q}dz\bigg)^{-1/{q}}\times|2^{j+1}B|^{1/q}\\
&=C\bigg(\int_{2^{j+1}B}|f(z)|^{p}w(z)^{p}dz\bigg)^{1/p}
\bigg(\int_{2^{j+1}B}w(z)^{q}dz\bigg)^{-1/{q}}\times|2^{j+1}B|^{(1-\alpha/n)-1/s}.
\end{split}
\end{equation*}

\textbf{Case 2:} $p>s'$. Applying H\"{o}lder's inequality with exponent $t=p/{s'}>1$ and the condition $w^{s'}\in A(p/{s'},q/{s'})$, we get
\begin{equation*}
\begin{split}
&\bigg(\int_{2^{j+1}B\backslash 2^jB}|f(z)|^{s'}dz\bigg)^{1/{s'}}\leq\bigg(\int_{2^{j+1}B}|f(z)|^{s'}dz\bigg)^{1/{s'}}\\
&\leq\bigg(\int_{2^{j+1}B}|f(z)|^{p}w(z)^{p}dz\bigg)^{\frac{1}{s'}\cdot\frac{s'}{p}}
\bigg(\int_{2^{j+1}B}w^{s'}(z)^{-\big(\frac{p}{s'}\big)'}dz\bigg)^{\frac{1}{s'\big(\frac{p}{s'}\big)'}}\\
&\leq C\bigg(\int_{2^{j+1}B}|f(z)|^{p}w(z)^{p}dz\bigg)^{1/p}\bigg(\int_{2^{j+1}B}w(z)^{q}dz\bigg)^{-1/{q}}
\times\big|2^{j+1}B\big|^{\frac{1}{s'\big(\frac{p}{s'}\big)'}+\frac{1}{q}}.
\end{split}
\end{equation*}
It follows directly from \eqref{cal3} that
\begin{equation*}
\frac{1}{s'\big(\frac{p}{s'}\big)'}+\frac{\,1\,}{q}=\frac{p-s'}{ps'}+\frac{\,1\,}{q}=\Big(1-\frac{\alpha}{n}\Big)-\frac{\,1\,}{s}.
\end{equation*}
Thus, in both cases \eqref{wang83} holds. Consequently, for any $x\in B$ and $y\in B$, we can deduce that
\begin{equation*}
\begin{split}
&\sum_{j=2}^\infty\int_{2^{j+1}B\backslash 2^jB}
\Big|\frac{\Omega(x-z)}{|x-z|^{n-\alpha}}-\frac{\Omega(y-z)}{|y-z|^{n-\alpha}}\Big||f(z)|\,dz\\
&\leq C\sum_{j=2}^\infty
\bigg[\frac{1}{2^{j-1}}
+\int_{\frac{|x-x_0|}{2^{j+1}r}}^{\frac{|x-x_0|}{2^jr}}\omega_s(\delta)\frac{d\delta}{\delta}
+\int_{\frac{|y-x_0|}{2^{j+1}r}}^{\frac{|y-x_0|}{2^jr}}\omega_s(\delta)\frac{d\delta}{\delta}\bigg]\\
&\times \bigg(\int_{2^{j+1}B}|f(z)|^{p}w(z)^{p}dz\bigg)^{1/p}\bigg(\int_{2^{j+1}B}w(z)^{q}dz\bigg)^{-1/{q}}\\
&\leq C\sum_{j=2}^\infty
\bigg[\frac{1}{2^{j-1}}+\int_{\frac{|x-x_0|}{2^{j+1}r}}^{\frac{|x-x_0|}{2^jr}}\omega_s(\delta)\frac{d\delta}{\delta}
+\int_{\frac{|y-x_0|}{2^{j+1}r}}^{\frac{|y-x_0|}{2^jr}}\omega_s(\delta)\frac{d\delta}{\delta}\bigg]\\
&\times\frac{w^q(2^{j+1}B)^{\kappa/p}}{w^q(2^{j+1}B)^{1/q}}\|f\|_{L^{p,\kappa}(w^p,w^q)}\\
&=C\bigg[1+2\int_0^1\omega_s(\delta)\frac{d\delta}{\delta}\bigg]\|f\|_{L^{p,\kappa}(w^p,w^q)},
\end{split}
\end{equation*}
where the last equality follows from the assumption that $\kappa=p/q$. Therefore, we obtain
\begin{equation*}
\begin{split}
\mathrm{II}&\leq\frac{1}{|B|}\int_{B}\bigg\{\frac{1}{|B|}\int_{B}\bigg(\sum_{j=2}^\infty\int_{2^{j+1}B\backslash 2^jB}
\Big|\frac{\Omega(x-z)}{|x-z|^{n-\alpha}}-\frac{\Omega(y-z)}{|y-z|^{n-\alpha}}\Big||f(z)|\,dz\bigg)dy\bigg\}dx\\
&\leq C\bigg[1+2\int_0^1\omega_s(\delta)\frac{d\delta}{\delta}\bigg]\|f\|_{L^{p,\kappa}(w^p,w^q)}.
\end{split}
\end{equation*}
Combining the above estimates for I and II and taking the supremum over all balls $B$, we conclude the proof of Theorem \ref{thm2}.
\end{proof}

By using the same arguments as in the proof of Theorems \ref{thm1} and \ref{thm2}, we can also show the following result. The proof needs appropriate but only minor modifications and we leave this to the interested reader.
\begin{thm}\label{thm3}
Suppose that $\Omega\in L^s(\mathbf{S}^{n-1})$ with $1<s\le\infty$. If $0<\alpha<n$, $s'\leq p<n/{\alpha}$, $1/q=1/p-{\alpha}/n$, $\kappa=p/q=1-{(\alpha p)}/n$ and $w^{s'}\in A(p/{s'},q/{s'})$, then the fractional maximal operator $M_{\Omega,\alpha}$ with homogeneous kernel is bounded from $L^{p,\kappa}(w^p,w^q)$ to $L^{\infty}(\mathbb R^n)$.
\end{thm}
In particular, when $\Omega\equiv1$ and $s=\infty$, as a direct consequence of Theorem \ref{thm3}, we have
\begin{cor}
If $0<\alpha<n$, $1\leq p<n/{\alpha}$, $1/q=1/p-{\alpha}/n$, $\kappa=p/q=1-{(\alpha p)}/n$ and $w\in A(p,q)$, then the fractional maximal operator $M_{\alpha}$ is bounded from $L^{p,\kappa}(w^p,w^q)$ to $L^{\infty}(\mathbb R^n)$.
\end{cor}

\bibliographystyle{elsarticle-harv}
\bibliography{<your bibdatabase>}

\end{document}